\documentclass[a4paper,10pt]{article}
\usepackage{a4}
\usepackage{amsmath,amsthm,amssymb,enumerate}
\usepackage{amsthm}

\usepackage{amsfonts, amsmath}   
\usepackage{amssymb}
\usepackage{amsmath}
\usepackage{graphics}

\usepackage{amsmath,amsthm,amssymb,enumerate}
\usepackage{amsthm}
\newtheorem{theorem}{Theorem}[section]

\newtheorem{prop}[theorem]{Proposition}

\newtheorem{obs}[theorem]{Observation}
\newtheorem{cor}[theorem]{Corollary}
\newtheorem{dfn}[theorem]{Definition}

\newtheorem{con}[theorem]{Conjecture}

\numberwithin{equation}{section}

\begin{document}
\title{\textbf{ Majority out-dominating sets in digraphs} }
\author{ {\bf \sc     Karam Ebadi$^\dag$,  Mart\'{\i}n Manrique$^*$,}\\ {\bf \sc Reza Jafary$^\dag$,  and J. Joseline Manora$^\ddag$ }\\\\ {\footnotesize  $^\dag$ Department of Mathematics, Islamic Azad University of Miandoab, Iran.}\\ {\footnotesize  {\it  karam\_ebadi@yahoo.com, jafary\_reza228@yahoo.com}} \\\\ {\footnotesize $^*$ $n$-CARDMATH, Kalasalingam University, India.} \\ {\footnotesize  {\it   martin.manrique@gmail.com}}\\\\ {\footnotesize $^\ddag$Department of Mathematics, T. B. M. L. College, India.}\\ {\footnotesize {\it    joseline\_manora@yahoo.co.in}}\\\\
}

\date{}
\maketitle

\begin{abstract}
The concept of majority domination in graphs has been defined in at least
two different ways: As a function  and as a set. In this work we
extend the latter concept to digraphs, while the former was extended in another paper.  Given a digraph $D=(V,A),$ a set $S\subseteq V$ is a \textit{ majority out-dominating set } (MODS)  of $D$ if $|N^+[S]|\geq \frac {n}{2}.$   The minimum cardinality of a  MODS in $D$ is
the {\it set majority out-domination number} $\gamma^+_{m}(D)$  of $D.$ In this work we introduce these concepts and prove some results about them, among which the characterization of minimal MODSs.
\end{abstract}



\section{Introduction}

Although domination and other related concepts have been extensively studied
for undirected graphs, the respective analogues on digraphs have not received
much attention. A survey of results on domination in directed graphs  is found in chapter 15 of \cite{h}, but it mostly focuses
on kernels and solutions (that
is, independent in- and out-dominating sets), and on domination
in tournaments.

The notion of majority out-dominating set is fairly interesting from the mathematical point of view, since it is close enough to that of out-dominating set as to inherit several of its properties and allow the adaptation of some known results, and at the same time it is different enough as to open a new line of research.

This concept has interesting applications, specially related to democracy: The main idea of democracy is that of a representative group which is accepted by a majority of the population. In some way, this corresponds to majority dominating sets in undirected graphs. However, it is important to notice that the relation is actually directed: The representative group must be accepted by at least half of the population, but if the group itself accepts or not a particular sector of such population has no influence at all in the scope of simple democracy. Of course, more complex systems exist, with the aim that every important minority has some acceptance from the representative group, and those systems are better fit for large populations, like that of a country. Nevertheless, simple democracy is still the best option for small groups, like the members of a club or those of a small company. In the context of simple democracy, the concept of  majority out-dominating set in digraphs works more accurately than that of majority dominating set in undirected graphs.

\section{Fundamentals}

Throughout this paper $D=(V,A)$ is a finite directed graph with neither
loops nor multiple arcs (but pairs of opposite arcs are allowed) and $%
G=(V,E) $ is a finite undirected graph with neither loops nor multiple
edges. Unless stated otherwise, $n$ denotes the {\it order} of $D$ (or $G$), that is, $n=|V|$. For basic terminology on graphs and digraphs we refer to \cite{cl}.

Let $G=(V,E)$ be a graph. For any vertex $u\in V$, the set $%
N_G(u)=\{v: uv\in E\}$ is called the
{\it neighborhood}  of $u$ in $G.$  $N_G[u]=N_G(u)\cup \{u\}$ is the {\it closed neighborhood} of $u$ in $G.$ The {\it degree} of $u$ in $G$ is $%
d_G(u)=|N_G(u)|.$ When the graph $G$ is clear from the context, we may write simply $N(u),$  $N[u],$ and $d(u).$

Let $D=(V,A)$ be a digraph. For any vertex $u\in V$, the sets $%
N^{+}_D(u)=\{v: uv\in A\}$ and $N^{-}_D(u)=\{v: vu\in A\}$ are called the
{\it out-neighborhood} and {\it in-neighborhood} of $u$ in $D,$ respectively. $N^{+}_D[u]=N^{+}_D(u)\cup \{u\}$ is the {\it closed out-neighborhood} of $u$ in $D,$ and $N^{-}_D[u]=N^{-}_D(u)\cup \{u\}$ is the {\it closed in-neighborhood} of $u$ in $D.$ When the digraph $D$ is clear from the context, we may write simply $N^+(u),$  $N^-(u),$ $N^+[u],$ and $N^-[u].$ The {\it out-degree} and  {\it in-degree} of $u$ in $D$ are defined by $%
d^{-}_D(u)=|N^{-}_D(u)|$ and $d^{+}_D(u)=|N^{+}_D(u)|,$ respectively. The  {\it maximum out-degree} of $D$ is
denoted by  $\Delta ^{+}_D.$
When the digraph $D$ is clear from the context, we may write $d^{-}(u),$ $d^{+}(u),$ and $\Delta ^{+}.$

Given a set $X\subseteq V$ and $u\in X,$ the set of {\it external private out-neighbors} of $u$ respect to $X$ is $pn^+(u,X)=\{v\in V\setminus X:N^-(v)\cap X=\{u\}\},$ and $pn^+[u,X]=pn^+(u,X)\cup \{u\}$. Moreover, $D[X]$ denotes the subdigraph of $D$ {\it induced} by $X.$\\

Let $G=(V,E)$ be a graph. A subset $S$ of $V$ is called a {\it dominating set} of $%
G$ if every vertex in $V\setminus S$ is adjacent to at least one vertex in $%
S.$ The minimum cardinality of a dominating set of $G$ is called the
{\it domination number} of $G$ and is denoted by $\gamma (G),$ or simply $\gamma .$

Let $D=(V,A)$ be a digraph. A subset $S$ of $V$ is called an {\it out-dominating
set} of $D$ if for every vertex $v\in V\setminus S$ there exists at least one
vertex $u\in S\cap N^{-}(v).$ The minimum cardinality of an out-dominating
set of $D$ is  the {\it out-domination number} of $D$ and is denoted by $%
\gamma ^{+}(D)$, or simply $\gamma ^{+}.$ {\it In-dominating sets} in digraphs are
defined in a similar way, and the minimum cardinality of an in-dominating
set of $D$ is called the {\it in-domination number} of $D,$ denoted by $%
\gamma ^{-}(D)$.\\

Let $G=(V,E)$ be a graph. A {\it majority dominating function} \cite{mf} is a function
$f:V\rightarrow \{-1,1\}$ such that the set $S=\{v\in V:\sum\limits_{u\in
N[v]}f(u)\geq 1\}~$satisfies $|S|~\geq \frac{n}{2};$ the {\it weight} of a
majority dominating function is $w(f)=\sum\limits_{v\in V}f(v),$ and $\min
\{w(f):f$ is a majority dominating function in $G\}$ is the {\it majority
domination number} of $G,$ denoted $\gamma _{maj}(G).$  A {\it majority dominating set} \cite{ms} is a set $M\subseteq V$ such that
$|N[M]|~\geq \frac{n}{2},$ and $\min \{|M|~:M$ is a majority dominating
set of $G\}$ is the {\it set majority domination number} of $G,$ denoted $\gamma
_{m}(G)$; a majority dominating set $M$ of $G$ such that $|M|~=\gamma _{m}(G)
$ is a $\gamma _{m}(G)$-{\it set}. It is straightforward that for
every graph $G$ and every majority dominating function $f$ of $G,$ $\gamma _{m}(G)\leq
|f^{-1}(1)|,$ since $f^{-1}(1)$ is a majority dominating set of $G.$

Both concepts can be naturally extended to digraphs: Given a digraph $D,$ a
{\it majority out-dominating function} of $D$ is a function $f:V\rightarrow
\{-1,1\}$ such that the set $S=\{v\in V:\sum\limits_{u\in N^{+}[v]}f(u)\geq
1\}~$satisfies $|S|~\geq \frac{n}{2};$ the {\it weight} of $f$ is $%
w(f)=\sum\limits_{v\in V}f(v),$ and $\min \{w(f):f$ is a majority
out-dominating function in $D\}$ is the {\it majority out-domination number} of $D,
$ denoted $\gamma _{maj}^{+}(D).$  Similarly, a {\it majority out-dominating set} (MODS) of $D$ is a set $%
M\subseteq V$ such that $|N^{+}[M]|~\geq \frac{n}{2},$ and $\min \{|M|~:M$
is a MODS of $D\}$ is the {\it set majority out-domination
number} of $D,$ denoted $\gamma _{m}^{+}(D)$; a MODS $M
$ of $D$ such that $|M|~=\gamma _{m}^{+}(D)$ is a $\gamma _{m}^{+}(D)$-{\it set}.

However, it does not hold that  for every digraph $D$ and every majority out-dominating function $f$ of $D,$ $\gamma _{m}^{+}(D)\leq |f^{-1}(1)|,$ since $f^{-1}(1)$ is a majority in-dominating set of $D$ (defined analogously), but not necessarily a
MODS of $D.$ For example, consider the digraph $D=(V,A)$ shown in Figure 1, where $V=\{u,v\}\cup S\cup T,$ $|S|=k\geq3,$ $|T|=k+2,$ $d^-(x)=0$ for every $x\in S\cup T,$ $d^+(u)=d^+(v)=0,$ $N^-(u)=S,$ and $N^-(v)=S\cup T.$ Then the function $g:V\rightarrow \{-1,1\}$ such that $g(u)=g(v)=1$ and $g(x)=-1$ for every $x\in S\cup T$ is a majority out-dominating function of $D$, but $\gamma^+_m(D)=k$.

\begin{center}
\unitlength 1mm 
\linethickness{0.4pt}
\ifx\plotpoint\undefined\newsavebox{\plotpoint}\fi 
\begin{picture}(71.5,48.25)(0,0)
\put(6.25,24.25){\circle*{2}}
\put(46,23.75){\circle*{2}}
\put(66.25,4){\circle*{2}}
\put(25.75,43.25){\circle*{2}}
\put(66.25,13.75){\circle*{2}}
\put(66.25,35.75){\circle*{2}}
\put(26.25,34){\circle*{2}}
\put(66.25,43){\circle*{2}}
\put(26,28.5){\circle*{1}}
\put(26,24.75){\circle*{1}}
\put(26,21.25){\circle*{1}}
\put(26,18.75){\circle*{1}}
\put(66,31.25){\circle*{1}}
\put(66,27){\circle*{1.12}}
\put(66,22.75){\circle*{1}}
\put(66.25,18.5){\circle*{1}}
\put(25.75,13.75){\circle*{2}}
\put(15.88,33.63){\vector(-1,-1){.07}}\multiput(25.5,43.25)(-.033712785,-.033712785){571}{\line(0,-1){.033712785}}
\put(16.13,29){\vector(-2,-1){.07}}\multiput(26,34)(-.066498316,-.033670034){297}{\line(-1,0){.066498316}}
\put(15.63,19){\vector(-2,1){.07}}\multiput(25.25,13.75)(-.061698718,.033653846){312}{\line(-1,0){.061698718}}
\put(35.88,33.38){\vector(1,-1){.07}}\multiput(25.75,43.25)(.034556314,-.033703072){586}{\line(1,0){.034556314}}
\put(36,28.75){\vector(2,-1){.07}}\multiput(26.25,34)(.0625,-.033653846){312}{\line(1,0){.0625}}
\put(35.63,18.63){\vector(2,1){.07}}\multiput(25.75,13.5)(.064967105,.033717105){304}{\line(1,0){.064967105}}
\put(56,33.38){\vector(-1,-1){.07}}\multiput(66,42.75)(-.035971223,-.033723022){556}{\line(-1,0){.035971223}}
\put(56.25,29.63){\vector(-3,-2){.07}}\multiput(66.25,35.5)(-.05730659,-.033667622){349}{\line(-1,0){.05730659}}
\put(55.88,19.13){\vector(-2,1){.07}}\multiput(65.75,14)(-.064967105,.033717105){304}{\line(-1,0){.064967105}}
\put(55.75,14.13){\vector(-1,1){.07}}\multiput(66,4.25)(-.034982935,.033703072){586}{\line(-1,0){.034982935}}
\put(5.5,20.25){$u$}
\put(25.5,45.25){$s_1$}
\put(25.5,35.75){$s_2$}
\put(25.5,10.5){$s_k$}
\put(45,20){$v$}
\put(68.25,43.25){$t_1$}
\put(68.5,35.5){$t_2$}
\put(68.75,14){$t_{k+1}$}
\put(68.5,3.5){$t_{k+2}$}
\end{picture}

Figure 1
\end{center}

In this article we focus on majority dominating sets, while we study majority out-dominating functions in another paper \cite{mafu}.

\section{ Majority out-dominating sets }

\begin{obs} \label{lema 0}
If $H$ is a spanning subdigraph of a digraph $D,$ then
 $\gamma^+_m(D)\leq \gamma^+_m(H).$
\end{obs}

\begin{proof}
The result follows immediately because  any MODS of $H$ is also a MODS of $D.$
\end{proof}

\begin{obs}\label{obs 1}
 For the directed path $P_{n}$, $\gamma
^{+}(P_{n})=\lceil \frac{n}{2}\rceil$, and for the directed cycle $C_{n}$ ($%
n\geq 3$), $\gamma ^{+}(C_{n})=\lceil \frac{n}{2}\rceil .$
\end{obs}

\begin{obs}\label{obs 2}
 For the directed path $P_{n}$, $\gamma^+_m(P_{n})=\lceil \frac{n}{4}\rceil$, and for the directed cycle $C_{n}$ ($%
n\geq 3$), $\gamma^+_m(C_{n})=\lceil \frac{n}{4}\rceil .$
\end{obs}

\begin{obs}\label{obs 0}
 For any digraph $D$  which has a hamiltonian circuit,  $\gamma^+_m(D)\leq\lceil \frac{n}{4}\rceil.$
\end{obs}

\begin{proof}
The result follows from Observation \ref{obs 2}.
\end{proof}

\begin{prop}  Let $l(D)$ denote the length of a longest directed path in $D$.  Then
$\gamma^+_m(D)\leq \lceil \frac {2n-l(D)-1}4\rceil,$ and the bound is sharp.
\end{prop}

\begin{proof}  Let $P$ be a longest directed path in $D.$   Let  $S_1$ be a minimum  MODS of  $P,$ and
let  $S_2\subseteq  V(D)\setminus V(P)$ such that $|S_2|=\lceil\frac{|V(D)\setminus V(P)|}2 \rceil.$ Clearly $S=S_1\cup S_2$  is a MODS of $D$ and hence $\gamma^+_m(D)\leq |S_1|+|S_2|=\lceil \frac{l(D)+1}{4}\rceil+\lceil\frac{n- l(D)-1}2 \rceil=\lceil \frac{2n-l(D)-1}{4} \rceil.$ The bound is trivially attained for directed paths.
\end{proof}

\begin{prop}  Let $c(D)$ denote the length of a longest directed cycle in $D$.
  Then $\gamma^+_m(D)\leq \lceil \frac {2n-c(D)}4\rceil,$ and the bound is sharp.
\end{prop}

\begin{proof}
The proof is similar to that of the previous proposition . The bound is trivially attained for directed cycles.
\end{proof}

\begin{theorem}\label{thm 3}
  For any digraph $D,$   $\gamma^+_m(D) = \gamma^+(D)$ if, and only if, $\Delta^+(D)= n-1.$
\end{theorem}

\begin{proof}
Let $D$ be a digraph with $\Delta^+(D)\leq n-2$; then  $\gamma^+(D)\geq 2.$ Let $S$ be a $\gamma^+$-set of $D,$ and let
$S=S_1\cup S_2$ with $|S_1|\geq1,\ |S_2|\geq1,$ and $S_1\cap S_2=\emptyset.$ Since $|N^+[S]|=n,$ then either $|N^+[S_1]|\geq \frac{n}2$ or $|N^+[S_2]|\geq \frac{n}2.$ If follows that at least one of $S_1$ and $S_2$ is a MODS of $D.$  Therefore, $ \gamma^+_m(D)< |S|= \gamma^+(D).$ The converse is obvious.
\end{proof}

\begin{prop}\label{thm 4}
  For any digraph $D,$   $ \gamma^+_m(D) \leq \lceil\frac{\gamma^+(D)}2\rceil.$ The bound is sharp.
\end{prop}

\begin{proof}
Let $S$ be a $\gamma^+$-set of $D.$ If $\Delta^+(D)= n-1,$ the result follows from Theorem \ref{thm 3}. Suppose $\Delta^+(D)\leq n-2.$ Then $|S|>1,$ so there are two disjoint non-empty sets $S_1$ and $S_2$ such that $S=S_1\cup S_2$ with $|S_1|=\lceil\frac{\gamma^+(D)}2 \rceil,\   |S_2|=\lfloor\frac{\gamma^+(D)}2\rfloor.$ Now, $N^+[S]=N^+[S_1]\cup N^+[S_2]$ implies that $n=|N^+[S]|\leq |N^+[S_1]|+ |N^+[S_2]|.$ Then either $|N^+[S_1]|\geq \lceil\frac{n}2\rceil$ or $|N^+[S_2]|\geq \lceil\frac{n}2\rceil,$ that is, at least one of $S_1$ and $S_2$ is a MODS of $D.$  Hence $ \gamma^+_m(D) \leq  \lceil\frac{\gamma^+(D)}2\rceil.$ Equality holds for directed paths and directed cycles, as follows from Observations \ref{obs 1} and \ref{obs 2}.
\end{proof}

\begin{obs}\label{obs 5}
  Let $D$ be a  digraph. Then   $ \gamma^+_m(D) =1$ if, and only if, there exists one vertex $v\in D$ such that
   $d^+(v)\geq \lceil\frac{n}2\rceil-1.$
\end{obs}

\begin{theorem}\label{thm 6}
  For any digraph $D$:
  \begin{enumerate}[(i)]
\item  $\lceil \frac{n}{2(\Delta^+(D)+1)} \rceil\leq \gamma^+_m(D).$ The bound is sharp.
 \item  Either $\gamma^+_m(D)=1$ or $\gamma^+_m(D)\leq \lceil \frac{n}2 \rceil-\Delta^+(D).$ In the second case the bound is sharp.
 \end{enumerate}
\end{theorem}

\begin{proof}
 $(i).$ Let $S=\{v_1,\dots,v_{\gamma^+_m}\}$ be a $\gamma^+_m$-set of $D.$ Then $\lceil \frac{n}2\rceil\leq |N^+[S]|\leq \sum\limits_{v\in S}d^+(v)+\gamma^+_m(D)\leq \sum\limits_{v\in S}\Delta^+(D)+\gamma^+_m(D)=\gamma^+_m(D)(\Delta^+(D)+1).$ Equality is attained by double stars oriented so that the stem vertices have in-degree zero.\\

 $(ii).$ Suppose $\gamma^+_m(D)>1.$ Then $\Delta^+(D)<\lceil\frac{n}2\rceil-1.$ Let $v$ be a vertex with  $d^+(v)=\Delta^+(D),$ and let $S\subseteq V(D)\setminus N^+(v)$ such that $v\in S$ and
 $|S|=\lceil \frac{n}2\rceil-\Delta^+(D).$  Since $v \in S$ and $S\cap N^+(v)=\emptyset,$ it follows that $S$ is a MODS of $D,$ so the result follows. Equality holds for the directed path $P_6$, among others.
\end{proof}

\begin{cor}\label{prop 7}
  For every digraph $D,$  $\gamma^+_m(D)\leq \frac{n-\Delta^+(D)+1}2.$ The bound is sharp.
\end{cor}
\begin{proof}
If $\Delta^+(D)=n-1,$ then $\gamma^+_m(D)=1=\frac{n-\Delta^+(D)+1}2.$ Otherwise, Theorem \ref{thm 6} $(ii)$ implies that $\gamma^+_m(D)\leq \frac{n-2\Delta^+(D)+1}2\leq \frac{n-\Delta^+(D)+1}2.$ As already mentioned, equality holds for any digraph $D$ such that $\Delta^+(D)=n-1.$
\end{proof}

\begin{prop}\label{prop 8}
Let $D$ be a digraph which is not a totally disconnected digraph
of odd order. If $S$ is a minimal MODS of $D,$ then $V\setminus S$ is a
MODS of $D.$
\end{prop}

\begin{proof}
Suppose $D$ is a totally disconnected (di)graph of even
order.  Then any minimal MODS of $D$ contains $\frac{n}2$ vertices
and hence its complement is also a MODS of $D.$
 Suppose $D$ is not a totally disconnected digraph, and let $S$ be a minimal
MODS of  $D.$ Therefore, $|S|\leq \frac{n}2$ and  $|V\setminus S|\geq \lceil\frac{n}2\rceil,$   which implies that
$V\setminus S$ is a MODS of $D.$
\end{proof}

\begin{theorem}
Let $S$ be a MODS of a digraph $D=(V,A).$ Then $S$ is
minimal if, and only if, one of the following conditions hold:

\begin{enumerate}[(i)]
\item $|N^{+}[S]|~>\left\lceil \frac{n}{2}\right\rceil $ and $\forall ~v\in
S,~|pn^+[v,S]|~>|N^{+}[S]|~-\left\lceil \frac{n}{2}\right\rceil .$

\item $|N^{+}[S]|~=\left\lceil \frac{n}{2}\right\rceil $ and $\forall ~v\in S,
$ either $v$ is an isolate in $D[S]$ or $pn^+(v,S)\neq \emptyset. $
\end{enumerate}
\end{theorem}
\begin{proof}
Let $D=(V,A)\ $be a digraph. Let $S$ be a minimal MODS
of $D$ and take $v\in S.$ Assume  $|N^{+}[S]|~>\left\lceil \frac{n}{2}%
\right\rceil .$ Since $S\setminus \{v\}$ is not majority out-dominating, $%
|N^{+}[S\setminus \{v\}]|~=|N^{+}[S]|~-|pn^+[v,S]|~<\left\lceil \frac{n}{2}%
\right\rceil .$ Hence $|pn^+[v,S]|~>|N^{+}[S]|~-\left\lceil \frac{n}{2}%
\right\rceil .$ Now assume $|N^{+}[S]|~=\left\lceil \frac{n}{2}\right\rceil .
$ Since $S\setminus \{v\}$ is not majority out-dominating, $%
|N^{+}[S\setminus \{v\}]|~<\left\lceil \frac{n}{2}\right\rceil .$ If $v$ is
an isolate in $D[S],$ $|N^{+}[S\setminus \{v\}]|~\leq \left\lceil \frac{n}{2}%
\right\rceil -1<\left\lceil \frac{n}{2}\right\rceil .$  If $v$ is not an
isolate in $D[S],$ it must forcibly have an external private neighbor, for
otherwise we would have $|N^{+}[S\setminus \{v\}]|~=\left\lceil \frac{n}{2}%
\right\rceil .$

Now let $S$ be a MODS of $D$ such that (i) or (ii)
hold, and take $v\in S.$ Assume  $|N^{+}[S]|~>\left\lceil \frac{n}{2}%
\right\rceil .$ Since $|N^{+}[S\setminus \{v\}]|~=|N^{+}[S]|~-|pn^+[v,S]|$ and
$|pn^+[v,S]|~>|N^{+}[S]|~-\left\lceil \frac{n}{2}\right\rceil ,$ it follows
that $|N^{+}[S]|~-|pn^+[v,S]|~<\left\lceil \frac{n}{2}\right\rceil ,$ which
implies that $S$ is minimal. Now assume $|N^{+}[S]|~=\left\lceil \frac{n}{2}%
\right\rceil .$ If $v$ is an isolate in $D[S],$ $|N^{+}[S\setminus
\{v\}]|~\leq \left\lceil \frac{n}{2}\right\rceil -1<\left\lceil \frac{n}{2}%
\right\rceil .$ If $v$ is not an isolate in $D[S]$ but $pn^+(v,S)\neq
\emptyset ,$ it holds as well that $|N^{+}[S\setminus \{v\}]|~\leq
\left\lceil \frac{n}{2}\right\rceil -1<\left\lceil \frac{n}{2}\right\rceil .$
Therefore, if $S$ is a MODS of $D$ such that (i) or
(ii) hold, then $S$ is minimal.
\end{proof}

We now consider the effect on $\gamma^+_m(D)$ of  the removal of a vertex or   an arc from $D.$

\begin{theorem} \label{theorem1}
 Let $D$ be any digraph with $\gamma^+_{m}(D)=k.$ Let $v \in V(D)$ and $e \in A(D).$ Then
\begin{enumerate} [(i)]
\item $ k\leq  \gamma^+_{m}(D-e) \leq k+1,$
\item $k-1 \leq \gamma^+_{m}(D-v) \leq\max \{k, k-1+d^+(v)\}.$
\end{enumerate}
\end{theorem}

\begin{proof} $(i)$  If $S$ is a $\gamma^+_{m}$-set of  $D-e,$  then  Observation \ref{lema 0} implies that
 $S$ is a MODS of $D.$ Hence  $k\leq \gamma^+_{m}(D-e).$

Now, let $S$ be a $\gamma^+_{m}$-set of $D$ and let $e=uv\in A$.  If $\{u,v\}\subseteq S,$ $\{u,v\}\subseteq V\setminus S,$ or $u\in V\setminus S$ and $v\in S,$ then $S$ is a MODS of $D-e.$
If  $u\in S$  and  $v\in V\setminus S,$  then $S'=S\cup\{v\}$ is  a MODS of $D-e.$
Thus  in all  cases  $\gamma^+_{m}(D-e) \leq k+1.$\\

 $(ii)$ Let $S$ be a $\gamma^+_{m}$-set of $D-v.$   Then
 $S'=S\cup\{v\}$ is a MODS of $D,$ so  $\gamma^+_{m}(D) \leq \gamma^+_{m}(D-v)+1.$  Thus $k-1\leq \gamma^+_{m}(D-v).$

 Now, let $S$ be a $\gamma^+_{m}$-set of $D.$  If $v\in V\setminus S,$ then $S$ is a MODS of $D-v.$   If $v\in S,$ then $S\cup N^+(v)$ is a MODS of $D-v.$   Thus, $\gamma^+_{m}(D-v)\leq\max \{k, k-1+d^+(v)\}.$
\end{proof}

Now we consider the effect on $\gamma^+_m(D)$ of   adding   an arc to $D.$

\begin{prop}
 Let $D$ be any digraph with $\gamma^+_{m}(D)=k,$  $e \in A(\overline{D}).$ Then $\gamma^+_{m}(D)-1 \leq \gamma^+_{m}(D+e)\leq\gamma^+_{m}(D).$
\end{prop}

\begin{proof}
 Let $e=uv \in A(\overline{D}).$ The fact that $\gamma^+_{m}(D+e) \leq\gamma^+_{m}(D)$ follows from Theorem \ref{theorem1}.  On the other hand, let $S$ be a $\gamma^+_{m}$-set of  $D+uv.$ Then $S\cup \{v\}$ is a $\gamma^+_{m}$-set of  $D,$ so $\gamma^+_{m}(D)-1 \leq \gamma^+_{m}(D+e).$
\end{proof}

\begin{prop}
Let $D$ be a digraph, and let $D'$ be the digraph obtained by reversing   the direction of a single arc of $D.$ Then
$|\gamma^+_m(D)-\gamma^+_m(D')|\leq1.$
\end{prop}

\begin{proof}
The proof is identical to the proof of the corresponding  result for out-domination given in  \cite{GDB}.
\end{proof}

\begin{dfn} Let $D=(V,A)$ be any digraph. An arc $e\in A(D)$ is $\gamma^+_{m}$-critical if $\gamma^+_{m}(D-e)=\gamma^+_{m}(D)+1.$
\end{dfn}

\begin{theorem}
An arc $e=uv$ of a digraph $D$ is $\gamma^+_{m}$-critical if, and only if, for every $\gamma^+_{m}(D)$-set $S$ we have that $u\in S,$ $v\in pn^+(u,S),$  and  $|N^+[S]|=\lceil\frac{n}2\rceil.$
\end{theorem}

\begin{proof}
Let $e=uv$ be a $\gamma^+_{m}$-critical arc, and let $S$ be a $\gamma^+_{m}(D)$-set. If $u \notin S,$ then $S$ is a $\gamma^+_{m}(D-e)$-set, which is a contradiction, so $u\in S.$ If $v\notin pn^+(u,S),$ we have that $S$ is a $\gamma^+_{m}(D-e)$-set, again a contradiction. Moreover, if $|N^+[S]|>\lceil\frac{n}2\rceil,$ then $S$ is as well a $\gamma^+_{m}(D-e)$-set. Therefore, for every $\gamma^+_{m}(D)$-set $S$ we have that $u\in S,$ $v\in pn^+(u,S),$  and  $|N^+[S]|=\lceil\frac{n}2\rceil.$

Conversely, suppose that for every $\gamma^+_{m}(D)$-set $S$ we have that $u\in S,$ $v\in pn^+(u,S),$  and  $|N^+[S]|=\lceil\frac{n}2\rceil.$ It follows that for every $\gamma^+_{m}(D)$-set $S,$ $|N_{D-e}^+[S]|= |N_D^+[S]|-1=\lceil\frac{n}2\rceil -1,$ so no $\gamma^+_{m}(D)$-set is a MODS of $D-e.$ Now suppose there is a $\gamma^+_{m}(D-e)$-set $S'$ with $|S'|=\gamma^+_{m}(D),$ then from Observation \ref{lema 0} it follows that $S'$ is a MODS of $D,$ which is a contradiction. Therefore, $e=uv$ is a $\gamma^+_{m}$-critical arc.
\end{proof}

\section{Oriented graphs}
Let $G=(V,E)$ be a graph. An orientation of $G$ is a digraph $D=(V,A)$ such that $uv\in E \Leftrightarrow (uv\in A$ or $vu\in A),$ and $|E|=|A|.$

\begin{center}
\unitlength 1mm 
\linethickness{0.4pt}
\ifx\plotpoint\undefined\newsavebox{\plotpoint}\fi 
\begin{picture}(60.5,27.75)(0,0)
\put(12.75,19){\circle*{1.5}}
\put(12.75,9){\circle*{1.5}}
\put(27,9){\circle*{1.5}}
\put(27,18.75){\circle*{1.5}}
\put(19.5,27){\circle*{1.5}}
\put(13,19.5){\vector(-3,-4){.07}}\multiput(18.75,26.5)(-.03362573,-.04093567){171}{\line(0,-1){.04093567}}
\put(12.5,18.5){\vector(0,-1){9}}
\put(13,9){\vector(1,0){13.5}}
\put(27,9.5){\vector(0,1){9}}
\put(19.5,26.75){\vector(-1,1){.07}}\multiput(26.75,19)(-.03372093,.03604651){215}{\line(0,1){.03604651}}
\put(45.5,19){\circle*{1.5}}
\put(45.5,9){\circle*{1.5}}
\put(59.75,9){\circle*{1.5}}
\put(59.75,18.75){\circle*{1.5}}
\put(52.25,27){\circle*{1.5}}
\put(45.75,19.5){\vector(-3,-4){.07}}\multiput(51.5,26.5)(-.03362573,-.04093567){171}{\line(0,-1){.04093567}}
\put(45.75,9){\vector(1,0){13.5}}
\put(52.25,26.75){\vector(-1,1){.07}}\multiput(59.5,19)(-.03372093,.03604651){215}{\line(0,1){.03604651}}
\put(59.75,18){\vector(0,-1){8.25}}
\put(45.25,9.75){\vector(0,1){8.25}}
\put(6.0,15){$D_1$}
\put(39.0,15){$D_2$}
\end{picture}

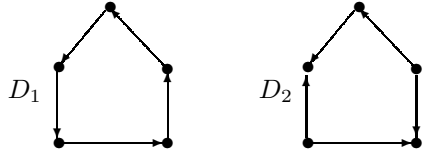
Figure 2: Two orientations of the cycle $C_5$
\end{center}

Two  distinct orientations of a given graph can have different majority domination numbers, as shown in  Figure 2 for the orientations $D_1$ and $D_2$  of the  cycle $C_5,$ where  $\gamma^{+}_m(D_1)=2$     and $\gamma^{+}_m(D_2)=1.$   This
suggests the following definitions:

\begin{dfn}
Let $G$ be a graph. The lower orientable set majority  domination number of $G$ is $ dom^{+}_m(G)=\min\{\gamma^{+}_m(D)\ :\ D\ {\rm is\ an\ orientation\ of}\ G \},$ and the   upper orientable set majority domination number of $G$ is $ DOM^{+}_m(G)=\max\{\gamma^{+}_m(D)\ :\ D\ {\rm is\ an\ orientation\ of}\ G \}.$
\end{dfn}

These concepts are inspired in the notions of lower orientable  domination number  $dom(G)$ and upper orientable  domination number $DOM(G)$, introduced by  Chartrand et al. in \cite{GDB}.

\begin{theorem}\label{thm 9}
For every graph $G,$ $dom^{+}_m(G)=\gamma_m(G).$
\end{theorem}

\begin{proof}
Let $G=(V,E)$ be a graph and let   $S$ be a minimum majority dominating set of $G.$  Now we consider the
following orientation $D=(V,A)$ of $G$: For every two adjacent vertices $u\in S$
and $v\in V\setminus S,$ let $uv\in A$;
edges between vertices of $S$ and edges between vertices of $V\setminus S$
can be oriented arbitrarily. Then $S$ is a  MODS of $D$ and so $dom^{+}_m(G)\leq \gamma^+_m(D)\leq |S|=\gamma_m(G).$

Now, let $D'$ be an orientation of $G$ for which $dom^{+}_m(G)= \gamma^+_m(D'),$  and let $S'$ be a $\gamma^+_m(D')$-set. It is clear that $S'$ is a  majority dominating set of $G.$ Therefore, $\gamma_m(G)\leq  |S'|=\gamma^+_m(D')=dom^{+}_m(G).$
\end{proof}

We now proceed to  determine the upper and lower orientable  set majority domination numbers for several classes of graphs:

\begin{prop}
\end{prop}
\begin{enumerate}[(i)]
\item  For $ n\geq1,$ we have $DOM^+_m(K_n)= dom^{+}_m(K_n)=1.$
\item For $ n\geq1,$ $dom^+_m(P_n)=\left\lceil \frac {n}6\right\rceil.$
 \item  For  $n\geq3, \
  dom^+_m(C_n)=\left\lceil \frac {n}6\right\rceil.$
\item For any two integers $r,s$ with $r\leq s,$
$ dom^{+}_m(K_{r,s}) =1.$

 \end{enumerate}

\begin{proof}

 (i)  Since for any tournament $T,$  $\sum\limits_{v\in V(T)}d^+(v)= \frac{(n)(n-1)}2 ,$ it follows that there exists a vertex $u$ in $T$ with
 $d^+(u)\geq \left\lceil \frac{n-1}{2}\right\rceil.$ This completes the proof.

 (ii) From $(i)$ of  Theorem \ref{thm 6} it follows that $dom^+_m(P_n)\geq\lceil \frac {n}6\rceil.$ Now, consider the following orientation $D=(V,A)$ of $P_n$: We number the vertices of $V(P_n)$ in order, that is, $V(P_n)=\{v_1,...,v_n\},$ with $N(v_i)=\{v_{i-1},v_{i+1}\}$ for $i\in\{2,...,n-1\},$ $N(v_1)=\{v_2\},$ and $N(v_n)=\{v_{n-1}\}$; we orient the edges of $P_n$ in such a way that for a vertex $v_i\in V,$ $d^+(v_i)=2$ if, and only if, $i \equiv 2 \ (\mod 3).$ Then $S=\{v_i:i\equiv 2 \ (\mod 3), \ i\leq \lceil \frac {n}2\rceil\}$ is a MODS of $D,$ and $|S|=\lceil \frac {n}6\rceil.$ Therefore, $dom^+_m(P_n)=\lceil \frac {n}6\rceil.$

  (iii) The proof is similar to that of (ii).

  (iv) Let $K_{r,s}=(V,E),$ and let $V_1=\{v_1,v_2,\dots,v_r\}$ and $V_2=\{u_1,u_2,\dots,u_s\}$ be the bipartition of $V.$  Let $D$ be an orientation of $G$ such that $d^-(v_1)=0.$  Then $\{v_1\}$ is a MODS of $D$, so  $dom^+_m(G)= 1.$
 \end{proof}

 \begin{theorem}\cite {GDB}\label{thm 011}
  For every integer $n\geq3,$
   $DOM(P_n)=DOM(C_n)=\left\lceil \frac {n}2\right\rceil.$
 \end{theorem}

 \begin{prop}
 For every integer $n\geq3,$
   $DOM^+_m(P_n)=DOM^+_m(C_n)=\left\lceil \frac {n}4\right\rceil.$
 \end{prop}

\begin{proof}
  The result follows from Observation \ref{obs 2},    Theorem \ref{thm 4}, and Theorem  \ref{thm 011}.
\end{proof}

 \begin{prop}
 For $n\geq3,$  $DOM^+_m(K_{1,n-1})=\lfloor \frac{n-1}2\rfloor.$
 \end{prop}

 \begin{proof}
 First notice that $\lfloor \frac{n-1}2\rfloor=\lceil \frac{n-2}2\rceil.$ Let $v$ be the central vertex of the star $K_{1,n-1}=(V,E),$ and take any orientation $D$ of $K_{1,n-1}.$ Let $X\subseteq V\setminus N^+_D[v],$ such that $|X|=\max \{0, \ \lceil \frac{n-2|N^+_D[v]|}2\rceil\}.$  If $d^+_D(v)>0,$  then the set $S=\{v\}\cup X$ is a MODS of $D,$ and $|S|\leq \lfloor \frac{n-1}2\rfloor.$ If $d^+_D(v)=0,$ then $X$ is a $\gamma^+_m(D)$-set of cardinality $\lfloor \frac{n-1}2\rfloor.$
\end{proof}

\begin{theorem}
 For every double star $G,$ $dom^+_m(G)=1.$ Moreover, if $n\geq5$ then $DOM^+_m(G)=2+\max\{0,\lceil\frac{n-8}2\rceil\}.$
 \end{theorem}

 \begin{proof}
 Let $G=(V,E)$ be a double star,  and let $u$ and $v$ be the stem vertices of $G,$ with $d(u)\leq d(v).$ If we take an orientation $D'$ of $G$ such that $d^-_{D'}(v)=0,$ then $\{v\}$ is a MODS of $D',$  so $dom^+_m(G)=1.$

 For $DOM^+_m(G),$ since there exist vertices $x\in N(u)\setminus \{v\}$ and $y\in N(v)\setminus \{u\},$ it follows that for any orientation $D$ of $G,$ either $u\in N^+_D(x)$ or $x\in N^+_D(u),$ and either $v\in N^+_D(y)$ or $y\in N^+_D(v).$ Therefore, there is always a set $S$ with $|S|=2$ and $|N^+[S]|\geq4.$ This means that for $n\leq8,$ $DOM^+_m(G)\leq2.$ Moreover, if $n>8$ then $S\cup X$ is a MODS of $D,$ where $X\subseteq V\setminus N^+[S]$ and $|X|=\max\{0,\lceil\frac{n-8}2\rceil\}.$ Then we have that for $n\geq5,$ $DOM^+_m(G)\leq 2+\max\{0,\lceil\frac{n-8}2\rceil\}.$

 On the other hand, if $D'$ is an orientation of $G$ such that $d^+(u)=1$ and $d^+(v)=0,$ then $\gamma^+_m(D')=2+\max\{0,\lceil\frac{n-8}2\rceil\}.$
  \end{proof}

\begin{obs}\label{obs 33}
For every graph $G=(V,E)$ with $n\leq4$ and such that $E\neq \emptyset,$ $DOM^{+}_m(G)=1.$
\end{obs}

\begin{prop}
Take two positive integers $r,s$ with $r\leq s,$
then $DOM^{+}_m(K_{r,s})=1$ if, and only if, $r+s\leq4$ or

    (i) $r=2,\ s= 3$

    (ii) $r=2,\ s= 4$

     (iii) $r=s=3$
  \end{prop}

  \begin{proof}
  Consider a complete bipartite graph $K_{r,s}.$ If  $r+s\leq4,$ from Observation \ref{obs 33} it follows that $DOM^{+}_m(K_{r,s})=1.$ Likewise, it is easy to check that in any orientation $D$ of $K_{2,3}, \ K_{2,4},$ and $K_{3,3},$ there is a vertex $v$ such that $d^+(v)\geq2,$ so $\{v\}$ is a MODS of $D.$

  Now take $K_{r,s}$ which is not one of the cases mentioned above, and let $R$ and $S$ be the defining partite sets of $V,$ with $|R|=r$ and $|S|=s.$ If $s>r+2,$ the orientation $D$ of $K_{r,s}$ such that $d^+(v)=0$ for every $v\in R$ satisfies $\gamma^+_m(D)>1.$ Now suppose $r\leq s\leq r+2.$ This implies $r>2,$ since otherwise the graph would be one of those mentioned earlier. Moreover, since $K_{3,3}$ is as well one of the cases already considered, we have that $r+s\geq7.$ Let $R=\{u_1,...,u_r\}$ and $S=\{v_1,...,v_s\},$ and consider the orientation $D=(V,A)$ of $K_{r,s}$ such that the only arcs whose tail is in $u_i$ are $\{u_iv_j:j=2i, \ j=2i-1\},$ where  the product is taken modulo $s.$ Then for every $u\in R,$ $d^+(u)=2.$ If $r\neq s,$  for every $v\in S,$ $d^+(v)\leq r-1$; since $r+s\geq 2r+1,$ this implies $\gamma^+_m(D)>1.$ If $r=s,$ then for every $v\in S,$ $d^+(v)= r-2$; since $r+s\geq 2r,$ we have as well that $\gamma^+_m(D)>1.$ Therefore,  $DOM^{+}_m(K_{r,s})>1$ except for the cases mentioned above.
  \end{proof}

  In general, it seems difficult to find $DOM^{+}_m(K_{r,s}).$ However, we have the following conjecture:

\begin{con}  \label{prop 2.1}
Let $K_{r,s}$ be a complete bipartite graph with $r\leq s,$ and such that $DOM^{+}_m(K_{r,s})\neq1.$ Then:

 \[DOM^{+}_m(K_{r,s}) = \left\lbrace
  \begin{array}{c l}
  2  & \text{if  $  s\leq r+2,$ }\\
    \lceil\frac{s-r}2\rceil  & \text{  otherwise.}
  \end{array}
\right. \]
\end{con}

\begin{theorem}
 For $n\geq 4,$  $dom^+_m(W_n)=1$ and $DOM^+_m(W_n)=\lceil  \frac {n-2}4 \rceil.$
\end{theorem}

\begin{proof}
Let $W_n=v+C_{n-1},$ and let $D$ be an orientation of $W_n$ such that $d^-(v)=0.$  Then $\{v\}$ is a an out-dominating set of $D$, which in particular is a MODS of $D,$ so  $dom^+_m(W_n)= 1.$

On the other hand, consider $W_n=v+C_{n-1}$ and let $D=(V,A)$ be any orientation of $W_n.$ Observe that for any two consecutive vertices $x,y$ of $C_{n-1},$  one of the arcs $xy$ and $yx$ is in $A.$  If $d^-_D(v)<\frac{n-1}2,$ then $\{v\}$ is a MODS of $D,$ as mentioned earlier. Suppose $d^-_D(v)\geq\frac{n-1}2,$ and number the vertices of $C_{n-1}$ following the order of the cycle, that is, $V(C_{n-1})=\{u_1,...,u_{n-1}\},$ in such a way that $u_1v\in A$ and $u_1u_2\in A$ (notice that such a vertex $u_1$ will always exist, since $d^-_D(v)\geq\frac{n-1}2$). Now from the set $S_1=\{u_6,u_7,u_8,u_9\}$ take one vertex which out-dominates other vertex of $S_1,$ and call it $z_1$; from the set $S_2=\{u_{10},u_{11},u_{12},u_{13}\}$ take one vertex which out-dominates other vertex of $S_2,$ and call it $z_2,$ and so on. The last set $S_k=\{...,u_{n-1}\}$ may have less than four vertices, but we take anyway a vertex $z_k$ which out-dominates another vertex of the set, unless $S_k=\{u_{n-1}\},$ in which case we take $z_k=u_{n-1}.$ Then $S=\{u_1\}\cup\{z_1,...,z_k\}$ is a MODS of $D$ of cardinality $\lceil  \frac {n-2}4 \rceil.$ If $D$ is the orientation of $W_n$ such that $d^+(v)=0$ and    $d^-(x)=1$ for every $x\in V\setminus \{v\}$ (i.e., $D-v$  is a directed cycle on $n-1$ vertices), then $S$ is a $\gamma^+_m(D)$-set.
\end{proof}

Finally, we note that an "Intermediate Value Theorem" for orientable majority out-domination holds:

\begin{theorem}
For every graph $G$ and every integer $c$ with $dom^{+}_m(G)\leq c\leq  DOM^+_m(G),$ there exists an orientation $D$ of $G$ such that $\gamma^+_m(D)=c.$
 \end{theorem}

 \begin{proof}
 The proof is identical to the proof of the corresponding  result for out-domination given in  \cite{GDB}.
 \end{proof}

\section{Conclusions and scope}
In this paper we extended the notion of majority dominating set to digraphs. In addition to its applications, the topic is of mathematical interest since the behavior of MODSs is somewhat different to that of their counterparts in graphs. This is only an introductory work, in which the concept is defined and some basic results are proven. We hope this paper will be helpful for people working in related topics, and perhaps it will encourage further research in the field.

It would be interesting to prove the NP-completness of the decision problem Majority Dominating Set (MODS): Given a graph $G$ and a positive integer $k$, does $G$ have a MODS of cardinality $k$ or less?

\end{document}